\documentclass[a4paper]{amsart}
\usepackage[T1]{fontenc}
\usepackage{amstext}
\usepackage{amsthm}
\usepackage{amssymb}
\usepackage{stackrel}
\usepackage[unicode=true,pdfusetitle,
 bookmarks=true,bookmarksnumbered=false,bookmarksopen=false,
 breaklinks=false,pdfborder={0 0 0},pdfborderstyle={},backref=false,colorlinks=false]
 {hyperref}

\makeatletter

\numberwithin{equation}{section}
\numberwithin{figure}{section}
\theoremstyle{plain}
\newtheorem{thm}{\protect\theoremname}[section]
\theoremstyle{remark}
\newtheorem{rem}[thm]{\protect\remarkname}
\theoremstyle{plain}
\newtheorem{cor}[thm]{\protect\corollaryname}
\theoremstyle{definition}
\newtheorem{convention}[thm]{\protect\conventionname}
\theoremstyle{definition}
\newtheorem{defn}[thm]{\protect\definitionname}
\theoremstyle{definition}
\newtheorem{warning}[thm]{\protect\warningname}
\theoremstyle{definition}
\newtheorem{notation}[thm]{\protect\notationname}
\theoremstyle{definition}
\newtheorem{construction}[thm]{\protect\constructionname}
\theoremstyle{plain}
\newtheorem{prop}[thm]{\protect\propositionname}
\theoremstyle{definition}
\newtheorem{recollection}[thm]{\protect\recollectionname}
\theoremstyle{plain}
\newtheorem{lem}[thm]{\protect\lemmaname}
\theoremstyle{definition}
\newtheorem{example}[thm]{\protect\examplename}

\usepackage{tikz-cd}
\usepackage{mathtools}
\usepackage{adjustbox}
\usepackage{enumitem}
\tikzcdset{scale cd/.style={every label/.append style={scale=#1},
    cells={nodes={scale=#1}}}}
\usepackage{quiver}
\usepackage{xcolor}
\usepackage{eucal}
\usepackage{mleftright}
\mleftright
\DeclareMathSymbol{:}{\mathpunct}{operators}{"3A}

\makeatother

\providecommand{\constructionname}{Construction}
\providecommand{\conventionname}{Convention}
\providecommand{\corollaryname}{Corollary}
\providecommand{\definitionname}{Definition}
\providecommand{\examplename}{Example}
\providecommand{\lemmaname}{Lemma}
\providecommand{\notationname}{Notation}
\providecommand{\propositionname}{Proposition}
\providecommand{\recollectionname}{Recollection}
\providecommand{\remarkname}{Remark}
\providecommand{\theoremname}{Theorem}
\providecommand{\warningname}{Warning}

\begin{document}
\global\long\def\sf#1{\mathsf{#1}}%

\global\long\def\cal#1{\mathcal{#1}}%

\global\long\def\bb#1{\mathbb{#1}}%

\global\long\def\bf#1{\mathbf{#1}}%

\global\long\def\u#1{\underline{#1}}%

\global\long\def\tild#1{\widetilde{#1}}%

\global\long\def\pr#1{\left(#1\right)}%

\global\long\def\abs#1{\left|#1\right|}%

\global\long\def\inp#1{\left\langle #1\right\rangle }%

\global\long\def\opn#1{\operatorname{#1}}%

\global\long\def\Set{\sf{Set}}%

\global\long\def\SS{\sf{sSet}}%

\global\long\def\Cat{\mathcal{C}\sf{at}}%

\global\long\def\Fin{\sf{Fin}}%

\global\long\def\Del{\mathbf{\Delta}}%

\global\long\def\id{\operatorname{id}}%

\global\long\def\-{\text{-}}%

\global\long\def\op{\mathrm{op}}%

\global\long\def\To{\Rightarrow}%

\global\long\def\rr{\rightrightarrows}%

\global\long\def\rl{\rightleftarrows}%

\global\long\def\mono{\rightarrowtail}%

\global\long\def\epi{\rightsquigarrow}%

\global\long\def\ot{\leftarrow}%

\global\long\def\lim{\operatorname{lim}}%

\global\long\def\colim{\operatorname{colim}}%

\global\long\def\holim{\operatorname{holim}}%

\global\long\def\hocolim{\operatorname{hocolim}}%

\global\long\def\Fun{\operatorname{Fun}}%

\global\long\def\Tw{\operatorname{Tw}}%

\global\long\def\Path{\operatorname{Path}}%

\global\long\def\Alg{\operatorname{Alg}}%

\global\long\def\mor{\operatorname{mor}}%

\global\long\def\Fact{\operatorname{Fact}}%

\global\long\def\Perm{\operatorname{Perm}}%

\global\long\def\weq{\mathrm{weq}}%

\global\long\def\fib{\mathrm{fib}}%

\global\long\def\inert{\mathrm{inert}}%

\global\long\def\act{\mathrm{act}}%

\global\long\def\cart{\mathrm{cart}}%

\global\long\def\xmono#1#2{\stackrel[#2]{#1}{\rightarrowtail}}%

\global\long\def\xepi#1#2{\stackrel[#2]{#1}{\rightsquigarrow}}%

\global\long\def\adj{\stackrel[\longleftarrow]{\longrightarrow}{\bot}}%

\global\long\def\btimes{\boxtimes}%

\global\long\def\t{\otimes}%

\global\long\def\S{\mathsection}%

\global\long\def\p{\prime}%

\global\long\def\pp{\prime\prime}%

\global\long\def\from{\colon}%

\global\long\def\ev{\operatorname{ev}}%

\title{Monoidal Relative Categories Model Monoidal $\infty$-Categories}
\begin{abstract}
We prove that the homotopy theory of monoidal relative categories
is equivalent to that of monoidal $\infty$-categories, and likewise
in the symmetric monoidal setting. As an application, we give
a concise and complete proof of the fact that every presentably
monoidal or presentably symmetric monoidal $\infty$-category
is presented by a monoidal or symmetric monoidal model category,
which, in the monoidal case, was sketched by Lurie, and in the
symmetric monoidal case, was proved by Nikolaus--Sagave.
\end{abstract}

\author{Kensuke Arakawa}
\address{Department of Mathematics, Kyoto University, Kyoto, 606-8502,
Japan}
\email{arakawa.kensuke.22c@st.kyoto-u.ac.jp}
\subjclass[2020]{18M05, 19D23, 18N55, 18N60, 55U35, 55P47,}

\maketitle
\tableofcontents{}

\section*{Introduction}


Many examples of nontrivial $\pr{\infty,1}$-categories, or $\infty$-categories
for short, arise from relative categories. Recall that a \textbf{relative
category} is a category $\cal C$ equipped with a subcategory
$\cal W\subset\cal C$ of \textbf{weak equivalences}, which contains
all isomorphisms of $\cal C$.\footnote{This definition deviates slightly from Barwick and Kan's definition
in \cite{BK12b}, where they only require that the subcategory
of weak equivalences contain all identity morphisms. But the
difference is minor: Every relative category in the sense of
Barwick and Kan can be made into a relative category in our sense
by adjoining all isomorphisms to the subcategory of weak equivalences,
and this does not affect the localization. Our definition is
better suited when dealing with functors that are defined up
to natural isomorphisms.} Starting from a relative category $\pr{\cal C,\cal W}$, one
can formally invert the morphisms in $\cal W$ to obtain an $\infty$-category
$\cal C[\cal W^{-1}]$, called the \textbf{localization} of $\cal C$
at $\cal W$. One reason for the success of Quillen's theory
of model categories \cite{Quillen_HA}, and more broadly, homotopical
methods in various areas of mathematics, rests on the fact that
the $\infty$-category $\cal C[\cal W^{-1}]$ carries a very
rich structure: For example, Dwyer and Kan showed that when $\cal C$
is a model category and $\cal W$ is the subcategory of weak
equivalences, then the mapping spaces of $\cal C[\cal W^{-1}]$
are exactly the derived mapping spaces \cite{DK80_3}. 

A natural question to ask is which $\infty$-categories can be
realized as a localization of a relative category. Dwyer and
Kan \cite[Theorem 2.5]{DK87} and Joyal \cite[13.6]{Joy_notes_qCat}
proved a striking result in this direction, showing that \textit{every}
$\infty$-category is a localization of a relative category.\footnote{Joyal's preprint does not contain a proof of this, and a proof
was published by Stevenson \cite{Ste2017}.} Later, Barwick and Kan built on this and showed that the homotopy
theory of relative categories is in fact \textit{equivalent}
to that of $\infty$-categories \cite{BK12b}. To be more precise,
they showed that the functor $\sf{RelCat}\to\Cat_{\infty}$ from
the (ordinary) category of relative categories to the $\infty$-category
of $\infty$-categories induces an equivalence of $\infty$-categories
\[
\sf{RelCat}[{\rm DK}^{-1}]\xrightarrow{\simeq}\Cat_{\infty}.
\]
Here ${\rm DK}$ denotes the subcategory of \textbf{DK-equivalences},
i.e., morphisms of relative categories that induce categorical
equivalences between the localizations.

The passage from relative categories to their localizations can
be adapted to monoidal categories: Define a \textbf{monoidal
relative category} to be a relative category $\pr{\cal C,\cal W}$
equipped with a monoidal structure on $\cal C$, such that morphisms
in $\cal W$ are stable under tensor products in $\cal C$. One
can show that the localization $\cal C[\cal W^{-1}]$ inherits
a monoidal structure in such a way that the localization $L:\cal C\to\cal C[\cal W^{-1}]$
is monoidal; moreover, the functor $L$ is initial among the
monoidal functors that invert morphisms in $\cal W$ \cite[Proposition 4.1.7.4]{HA}. 

Given this, it is natural to ask which monoidal $\infty$-categories
arise as a localization of a monoidal relative category. Better
yet, we should ask how the homotopy theory of monoidal relative
categories is related to that of monoidal $\infty$-categories.
The main theorem of this paper answers these questions:
\begin{thm}
[Theorem {\ref{thm:main}}]\label{thm:intro}Monoidal localization
determines an equivalence of $\infty$-categories
\[
\sf{MonRelCat}[{\rm DK}^{-1}]\xrightarrow{\simeq}\cal M\sf{on}\Cat_{\infty},
\]
where $\sf{MonRelCat}[{\rm DK}^{-1}]$ denotes the localization
of the category of monoidal relative categories at the monoidal
functors that are DK-equivalences, and $\cal M\sf{on}\Cat_{\infty}$
denotes the $\infty$-category of monoidal $\infty$-categories.
A similar result holds in the symmetric monoidal case, too.
\end{thm}

\begin{rem}
According to \cite[Theorem 6.1 (iv)]{BK12b}, every $\infty$-category
is a localization of a relative \textit{poset}. We do not know
if this is true monoidally, i.e., if every monoidal $\infty$-category
is a localization of a relative monoidal poset. (The symmetric
monoidal analog is false. For example, if $\mathcal{C}$ is a
symmetric monoidal category which is not equivalent to a strict
symmetric monoidal category, then $\cal C$ can never be a symmetric
monoidal localization of a symmetric monoidal relative poset.)
\end{rem}

Joyal's delocalization theorem and Barwick--Kan's theorem come
in handy when one wants to prove generic statements about $\infty$-categories,
as they allow us to reduce the statements to those of ordinary
categories. (See \cite{Ste2017, Arl20, JacoMT} for example.)
Similarly, Theorem \ref{thm:intro} is useful in proving statements
about monoidal $\infty$-categories. As an example of this, we
will see that the following well-known result follows immediately
from Theorem \ref{thm:intro}:
\begin{cor}
[Theorem {\ref{thm:NS17}}]\label{cor:intro2}Every presentably
monoidal $\infty$-category is presented by a combinatorial monoidal
model category. A similar result holds in the symmetric monoidal
case, too.
\end{cor}

\begin{rem}
A proof of Corollary \ref{cor:intro2} was sketched in \cite[Remark 4.1.8.9]{HA},
but it seems that the detail of the proof has never been given.
Later, the author learned on MathOverflow \cite{488608} that
the details can be filled by Ramzi's recent work on the monoidal
Grothendieck construction \cite{Ram22}. In the symmetric monoidal
case, Corollary \ref{cor:intro2} was established by Nikolaus
and Sagave \cite{NS17}. Both Lurie and Nikolaus--Sagave use
techniques that are different from ours.
\end{rem}

We conclude this introduction by sketching our strategy for Theorem
\ref{thm:intro}. The starting point is Thomason's theorem \cite{Tho95}
and its refinement by Mandell \cite{Man10}. Thomason's theorem
asserts that every grouplike $E_{\infty}$-algebra (or infinite
loop spaces in spaces) arises from a symmetric monoidal category.
Mandell later refined this by showing that the homotopy theory
of $E_{\infty}$-algebras in spaces is equivalent to a localization
of the category of symmetric monoidal categories at the functors
inducing equivalences between classifying spaces. Theorem \ref{thm:intro}
may be regarded as a natural generalization of these theorems,
where categories are replaced by relative categories and spaces
by $\infty$-categories. Luckily, Mandell's arguments can mostly
be adapted to our setting, and we will prove Theorem \ref{thm:intro}
by making the necessary changes.

\subsection*{Outline of the Paper}

In Section \ref{sec:Man10}, we present a modification of Mandell's
argument \cite{Man10} and establish an equivalence between the
homotopy theory of permutative relative categories and that of
functors $\Fin_{\ast}\to\sf{RelCat}$ satisfying the Segal condition.
We then use this to prove Theorem \ref{thm:intro} in Section
\ref{sec:main}. As an application, we prove Corollary \ref{cor:intro2}
in Section \ref{sec:appl}. 

In the appendices, we collected some standard materials that
are well-known to the experts but are hard to find in the literature.
There are two appendices, one on relative categories and the
other on relative cartesian fibrations. These appendices should
be consulted as the need arises.

While we stated most of our results above for monoidal cases,
we will mainly be concerned with the symmetric monoidal cases,
as the latter needs more care than the former. It is straightforward
to adapt the arguments to the monoidal case, and we will leave
this task to the reader.

\subsection*{Notation and Conventions}
\begin{convention}
By an \textit{$\infty$-category}, we mean \textit{quasi-categories}
as developed by Joyal and Lurie \cite{Joyal_qcat_Kan,HTT}. We
will not notationally distinguish between ordinary categories
and their nerves. Unless stated otherwise, we will follow Lurie's
books \cite{HTT,HA} in terminology and notation.
\end{convention}

\begin{convention}
For various notions associated with $2$-categories (lax natural
transformations, lax colimit, etc.), we follow \cite{JY21}.
\end{convention}

\begin{defn}
We will write $\Fin_{\ast}$ for the category of the finite based
sets $\inp n=\pr{\{\ast,1,\dots,n\},\ast}$, $n\geq0$, and based
maps. A morphism $f:\inp n\to\inp m$ is said to be \textbf{inert}
if for each $1\leq i\leq m$, the inverse image $f^{-1}\pr i$
consists of exactly one element; if further the induced map $\u m\to\u n=\{1,\dots,n\}$
is order-preserving, we say that $f$ is \textbf{strongly inert}.
We say that $f$ is \textbf{active} if it carries the set $\u n$
into $\u m$. We will often depict an inert morphism by $\mono$
and active morphisms by $\rightsquigarrow$. Note that $f$ factors
uniquely as $f=f_{\act}f_{\inert}$, where $f_{\inert}$ is strongly
inert and $f_{\act}$ is active. 

For each $n\geq0$ and each $S\subset\u n$, we let $\rho^{S}:\inp n\to\inp{\abs S}$
denote the strongly inert map that carry each element $i\in\u n\setminus S$
to the base point. In the case where $S=\{i\}$ is a singleton,
we will write $\rho^{S}=\rho^{i}$. 
\end{defn}

\begin{rem}
\label{rem:nabla}When dealing with monoidal categories and monoidal
$\infty$-categories, we customarily replace the category $\Fin_{\ast}$
by the opposite of the category $\Del$ of finite nonempty ordinals
and poset maps. In this case, inert maps correspond to subinterval
inclusions, while active maps corresponds to maps that preserve
the minimum and maximum elements.

For the purpose of this paper, it is more convenient to replace
the category $\Del^{\op}$ by another isomorphic category $\nabla$,
which is defined as follows:
\begin{itemize}
\item Objects are the linearly ordered sets $[[n]]=\{-1<0<\dots<n\}$
where $n\geq0$.
\item Morphisms are the poset maps $[[n]]\to[[m]]$ preserving minimal
and maximal elements.
\end{itemize}
An isomorphism of categories $\phi\from\Del^{\op}\xrightarrow{\cong}\nabla$
is given in the following way: The poset $[[n]]$ can be identified
with the poset of downward-closed subsets of $[n]$, ordered
by inclusion. Explicitly, we identify each $x\in[[n]]$ with
the subset $\{i\in[n]\mid i\leq x\}\subset[n]$. With this identification,
we can associate to each poset map $u\from[n]\to[m]$ a map $\phi\pr u\from[[m]]\to[[n]]$,
given by $S\mapsto u^{-1}\pr S$, and this defines $\phi$. The
fact that $\phi$ is an isomorphism of categories follows from
the observation that, a poset map $[n]\to[m]$ determines and
is determined by a sequence
\[
\emptyset=S_{-1}\subset S_{0}\subset\cdots\subset S_{m}=[n]
\]
of downward-closed subsets of $[n]$. (Given such a sequence,
the corresponding map $[n]\to[m]$ carries $S_{i}\setminus S_{i-1}$
to $i\in[m]$.) 

Note that inert maps of $\Del^{\op}$ correspond to those maps
$[[m]]\to[[n]]$ such that, each non-extremum element of $[[m]]$
has a unique inverse image; active maps of $\Del^{\op}$ corresponds
to those maps $[[m]]\to[[n]]$ such that the preimages of extremum
elements are singleton.
\end{rem}

\begin{defn}
We let $\cal{SMC}\sf{at}_{\infty}$ denote the localization of
the ordinary category of symmetric monoidal $\infty$-categories
and symmetric monoidal functors (in the sense of \cite[Definition 2.0.0.7]{HA})
at the categorical equivalences. (For a concrete model of this
$\infty$-category, see \cite[Variant 2.1.4.13]{HA}.) 
\end{defn}

\begin{warning}
Unlike plain categories, we make a clear distinction between
an ordinary symmetric monoidal category $\cal C$ (as defined
in \cite[Chapter XI]{Mac98}) and the associated Grothendieck
opfibration $\cal C^{\t}\to\Fin_{\ast}$
\end{warning}

\begin{defn}
Let $\cal C$ be a monoidal category, and let $S=\{s_{1}<\dots<s_{n}\}$
be a finite totally ordered set (such as subsets of $\bb Z$).
We will write $\bigotimes_{S}:\cal C^{S}\to\cal C$ for the functor
given by $\pr{C_{s}}_{s\in S}\mapsto\pr{\cdots\pr{C_{s_{1}}\otimes C_{s_{2}}}\otimes\cdots}\otimes C_{s_{n}}$.
(When $S$ is empty, we interpret $\bigotimes_{S}$ as the constant
functor at the monoidal unit; when $S$ is a singleton, we interpret
$\bigotimes_{S}$ as the identity functor of $\cal C$.) We use
notations such as $\bigotimes_{i=1}^{n}$ in a similar way.
\end{defn}

\begin{defn}
\label{def:relcat}\textbf{Relative functors} of relative categories
are functors of underlying categories that preserve weak equivalences.
We let $\sf{RelCat}$ denote the $2$-category of relative categories,
relative functors, and natural weak equivalences (i.e., natural
transformations whose components are weak equivalences). Following
\cite[1.2]{BK12b}, we say that a relative functor $f:\cal C\to\cal D$
is a \textbf{homotopy equivalence} if there is a relative functor
$g:\cal D\to\cal C$ such that $gf$ and $fg$ are connected
by a zig-zag of natural weak equivalences. 

We will also make occasional use of \textbf{relative $\infty$-categories},
which are pairs $\pr{\cal C,\cal W}$ where $\cal C$ is an $\infty$-category
and $\cal W\subset\cal C$ is a subcategory containing all equivalences
of $\cal C$. We will write $L\pr{\cal C}=L\pr{\cal C,\cal W}=\cal C[\cal W^{-1}]$
for the localization of $\cal C$ at $\cal W$. Unless stated
otherwise, we will identify an $\infty$-category $\cal C$ with
the relative $\infty$-category whose weak equivalences are the
equivalences.
\end{defn}

\begin{convention}
By \textbf{monoidal functors} of monoidal categories, we will
always mean \textit{strong} monoidal functors in the sense of
\cite[Chapter XI]{Mac98}. We follow a similar convention for
symmetric monoidal functors. We say that a monoidal functor is
\textbf{strict} if its structure natural transformations are
the identity natural transformations.
\end{convention}

\begin{notation}
We let $\sf{SMCat}$ denote the category of small symmetric monoidal
categories and symmetric monoidal functors, and let $\sf{PermCat}\subset\sf{SMCat}$
denote the subcategory spanned by the permutative categories
and strict symmetric monoidal functors. (Recall that a symmetric
monoidal category is called a \textbf{permutative category} if
its underlying monoidal category is strict \cite[Definition 4.1]{May74}.)
\end{notation}

\begin{defn}
A \textbf{symmetric monoidal relative category} is a relative
category $\cal C$ equipped with a symmetric monoidal structure,
such that the tensor bifunctor $\otimes:\cal C\times\cal C\to\cal C$
is a relative functor. If the underlying symmetric monoidal category
is permutative, we say that $\cal C$ is a\textbf{ permutative
relative category}. We will write $\sf{SMRelCat}$ for the category
of symmetric monoidal relative categories and symmetric monoidal
functors whose underlying functors are relative functors. We
also write $\sf{PermRelCat}$ for the category of permutative
relative categories and strictly symmetric monoidal relative
functors.
\end{defn}

\begin{defn}
A functor $F:\Fin_{\ast}\to\sf{RelCat}$ is said to satisfy the
\textbf{Segal condition} if for each $n\geq0$, the inert maps
$\{\rho^{i}:\inp n\to\inp 1\}_{1\leq i\leq n}$ induces a DK-equivalence
$F\inp n\xrightarrow{\simeq}\prod_{1\leq i\leq n}F\inp 1$. We
will write $\Fun^{{\rm Seg}}\pr{\Fin_{\ast},\sf{RelCat}}\subset\Fun\pr{\Fin_{\ast},\sf{RelCat}}$
for the full subcategory spanned by the functors $\Fin_{\ast}\to\sf{RelCat}$
satisfying the Segal condition.
\end{defn}

\section{\label{sec:Man10}Variations of Mandell's Constructions}

In this section, we construct a pair of functors
\[
\Fact:\sf{PermRelCat}\rl\Fun^{{\rm Seg}}\pr{\Fin_{\ast},\sf{RelCat}}:\Perm
\]
(Constructions \ref{const:Fact} and \ref{const:Perm}). We can
regard $\sf{PermRelCat}$ as a relative category whose weak equivalences
are those whose images in $\sf{RelCat}$ are DK-equivalences;
we can also view $\Fun^{{\rm Seg}}\pr{\Fin_{\ast},\sf{RelCat}}$
as a relative category with weak equivalences given by natural
DK-equivalences. We will show that the functors above will be
relative functors for these structures of relative categories.
We then prove the following theorem:
\begin{thm}
\label{thm:Man10}The functors $\Fact$ and $\Perm$ are homotopy
equivalences of relative categories, which are homotopy inverses
of each other.
\end{thm}

The proof of Theorem \ref{thm:Man10} will be a modification
of Mandell's work \cite{Man10}. 

\subsection{\label{subsec:fact}The Functor \texorpdfstring{$\Fact$}{Fact}}

Given a symmetric monoidal relative category $\cal C$, we can
define a pseudofunctor $\cal C^{\bullet}:\Fin_{\ast}\to\sf{RelCat}$
by mapping each morphism $f:\inp n\to\inp m$ to the functor
\[
\cal C^{n}\to\cal C^{m},\,\pr{C_{i}}_{1\leq i\leq n}\mapsto\pr{\bigotimes_{i\in u^{-1}\pr j}C_{i}}_{1\leq j\le m}.
\]
The reason why this fails to be a strict functor is that tensor
products are in general not associative (or unital or commutative)
on the nose. Our functor $\Fact$ gives a rectification of this
pseudofunctor up to DK-equivalence, by keeping track of \textit{all}
possible ways to tensor objects. The construction is originally
due to May \cite{May78} (there attributed to Segal).
\begin{construction}
\label{const:Fact}We define a colored operad $\Fact_{n}$ of
\textbf{$n$-factorizations} as follows: Its colors are the subsets
$S\subset\{1,\dots,n\}$. There is a multi-morphism $\pr{S_{1},\dots,S_{k}}\to T$
if and only if $T$ is a disjoint union of the sets $S_{1},\dots,S_{k}$,
in which case it is unique. 

If $\cal C$ is a symmetric monoidal relative category, we let
$\Fact_{n}\pr{\cal C}$ denote the full subcategory of $\Alg_{\Fact_{n}}\pr{\cal C}$
spanned by the $\Fact_{n}$-algebras $A$ with the following
property: For each multi-morphism $\pr{S_{1},\dots,S_{n}}\to T$,
the map
\[
A\pr{S_{1}}\otimes\cdots\otimes A\pr{S_{n}}\to A\pr T
\]
is a weak equivalence. We will regard $\Fact_{n}\pr{\cal C}$
as a relative category whose weak equivalences are the maps $X\to Y$
such that, for each $S\in\Fact_{n}$, the map $X\pr S\to Y\pr S$
is a weak equivalence.

Every map $f:\inp n\to\inp m$ of pointed sets induces a map
$\Fact_{m}\to\Fact_{n}$ of colored operads, given on objects
by $S\mapsto f^{-1}\pr S$. Pulling back along this map, we obtain
a relative functor $\Fact_{n}\pr{\cal C}\to\Fact_{m}\pr{\cal C}$.
This makes the collection $\{\Fact_{n}\pr{\cal C}\}_{\inp n\in\Fin_{\ast}}$
into a functor $\Fin_{\ast}\to\sf{RelCat}$, thereby giving rise
to a functor
\[
\Fact:\sf{SMRelCat}\to\Fun\pr{\Fin_{\ast},\sf{RelCat}}.
\]
\end{construction}

\begin{rem}
In the monoidal case, we replace $\Fact_{n}$ by the non-symmetric
colored operad whose colors are the subsets $S\subset\underline{n}$
that are \textit{convex} in the following sense: If $i,j\in S$
and $i<j$, then $\{k\in\underline{n}\mid i\leq k\leq j\}\subset S$.
Multiarrows are defined as in the symmetric monoidal case.
\end{rem}

\begin{prop}
\label{prop:Seg}Let $\cal C$ be a symmetric monoidal relative
category. 
\begin{enumerate}
\item For each $n\geq0$, the forgetful functor
\[
\Phi=\Phi_{n}:\Fact_{n}\pr{\cal C}\to\cal C^{n},\,X\mapsto\pr{X\pr{\{i\}}}_{1\leq i\leq n},
\]
is a homotopy equivalence of relative categories.
\item The functors of (1) are part of a lax natural transformation
$\Fact\pr{\cal C}\To\cal C^{\bullet}$ of pseudofunctors $\Fin_{\ast}\to\sf{RelCat}$,
which is natural in $\cal C\in\sf{SMRelCat}$. 
\item The functor $\Fact\pr{\cal C}:\Fin_{\ast}\to\sf{RelCat}$ satisfies
the Segal condition.
\end{enumerate}
\end{prop}

\begin{proof}
For part (1), define a functor $\Psi:\cal C^{n}\to\Fact_{n}\pr{\cal C}$
by mapping each object $\pr{X_{1},\dots,X_{n}}\in\cal C^{n}$
to the $\Fact_{n}$-algebra $S\mapsto\bigotimes_{s\in S}X_{i}$,
with structure maps $\bigotimes_{1\leq i\leq n}\bigotimes_{s\in S_{i}}X_{s}\to\bigotimes_{s\in S_{1}\cup\dots\cup S_{n}}X_{s}$
provided by the coherence isomorphisms of $\cal C$. We then
have $\Phi\circ\Psi=\id_{\cal C^{n}}$. Also, for each $A\in\Fact_{n}\pr{\cal C}$,
the maps
\[
\{\bigotimes_{s\in S}A\pr{\{s\}}\to A\pr S\}_{S\subset\{1,\dots,n\}}
\]
determine a weak equivalence of $\Fact_{n}$-algebras $\Psi\circ\Phi\pr A\xrightarrow{\simeq}A$.
This weak equivalence is natural in $A$, so we have shown that
$\Phi$ and $\Psi$ are homotopy inverses of each other, proving
(1).

For part (2), let $u:\inp n\to\inp m$ be a morphism of $\Fin_{\ast}$.
There is a natural transformation depicted as 
\[\begin{tikzcd}
	{\operatorname{Fact}_n(\mathcal{C})} & {\mathcal{C}^n} \\
	{\operatorname{Fact}_m(\mathcal{C})} & {\mathcal{C}^m}
	\arrow[""{name=0, anchor=center, inner sep=0}, "{\Phi_n}", from=1-1, to=1-2]
	\arrow["{\operatorname{Fact}(\mathcal{C})(u)}"', from=1-1, to=2-1]
	\arrow["{\mathcal{C}^\bullet (u)}", from=1-2, to=2-2]
	\arrow[""{name=1, anchor=center, inner sep=0}, "{\Phi_m}"', from=2-1, to=2-2]
	\arrow[shorten <=4pt, shorten >=4pt, Rightarrow, from=0, to=1]
\end{tikzcd}\]whose component at an object $A\in\Fact_{n}\pr{\cal C}$ is provided
by the maps 
\[
\{\bigotimes_{i\in u^{-1}\pr j}A\pr{\{i\}}\xrightarrow{\simeq}A\pr{u^{-1}\pr j}\}_{1\leq j\leq m}.
\]
These natural weak equivalences determine a lax natural transformation
$\Fact\pr{\cal C}\To\cal C^{\bullet}$, and this lax natural
transformation is natural in $\cal C$.

For part (3), we must show that for each $n\geq0$, the map
\[
\Fact_{n}\pr{\cal C}\to\prod_{1\leq i\leq n}\Fact_{1}\pr{\cal C}
\]
is a DK-equivalence. Since the natural transformation appearing
in (2) associated to each morphism of $\Fin_{\ast}$ is a natural
weak equivalence, by part (1) we are reduced to showing that
the map $\cal C^{n}\to\prod_{1\leq i\leq n}\cal C$ is a DK-equivalence,
which is clear. The proof is now complete.
\end{proof}

\subsection{\label{subsec:perm}The Functor \texorpdfstring{$\Perm$}{Perm}}

We next construct a functor 
\[
\Perm:\Fun\pr{\Fin_{\ast},\sf{RelCat}}\to\sf{PermRelCat}
\]
whose restriction will be a homotopy inverse to the functor $\Fact$
of Construction \ref{const:Fact}. The functor $\Perm$ is a
generalization of the \textit{inverse $K$-theory functor }of
Mandell \cite{Man10} to the setting of relative categories.

Before we describe the construction, let us start with a motivation.
We are generally interested in whether a functor $F:\Fin_{\ast}\to\cal X$
satisfies the Segal condition, where $\cal X$ is a category
equipped with finite products and an appropriate notion of ``weak
equivalences''. Recall that this condition says that, for each
active map $u:\inp n\epi\inp m$ in $\Fin_{\ast}$, the map
\[
F\inp n\to\prod_{i=1}^{m}F\inp{n_{i}}
\]
is a weak equivalence, where for each $1\leq i\leq m$ we factored
$\rho^{i}f$ as a strongly inert map $\inp n\to\inp{n_{i}}$
followed by an active map $\inp{n_{i}}\to\inp 1$. The product
on the right-hand side is covariantly functorial in $\inp n$
\textit{and} contravariantly functorial in $\inp m$. Such functoriality
can be encoded as ordinary functoriality, using the \textit{twisted
arrow categor}y:
\begin{recollection}
\label{rec:twist}Let $\cal C$ be a category. The (right) \textbf{twisted
arrow category} $\Tw\pr{\cal C}$ is the (contravariant) Grothendieck
construction of the hom-functor $\cal C^{\op}\times\cal C\to\sf{Cat}$.
So its objects are the morphisms of $\cal C$, and a morphism
$\pr{f:X\to Y}\to\pr{g:Z\to W}$ in $\Tw\pr{\cal C}$ is a pair
$\pr{u,v}$ of morphisms in $\cal C$ rendering the diagram 
\[\begin{tikzcd}
	X & Z \\
	Y & W
	\arrow["u", from=1-1, to=1-2]
	\arrow["f"', from=1-1, to=2-1]
	\arrow["g", from=1-2, to=2-2]
	\arrow["v", from=2-2, to=2-1]
\end{tikzcd}\]commutative.
\end{recollection}

\begin{construction}
\label{const:Tw}We define a functor
\[
\pr -^{\Tw}:\Fun\pr{\Fin_{\ast},\sf{RelCat}}\to\Fun\pr{\Tw\pr{\Fin_{\ast}^{\act}},\sf{RelCat}}
\]
as follows. Given a functor $F:\Fin_{\ast}\to\sf{RelCat}$, the
functor $F^{\Tw}:\Tw\pr{\Fin_{\ast}^{\act}}\to\sf{RelCat}$ is
defined on objects by
\[
F^{\Tw}\pr{\inp n\epi\inp m}=\prod_{i=1}^{m}F\inp{n_{i}},
\]
where for each $1\leq i\leq m$, we factored $\rho^{i}f$ as
a strongly inert map $\inp n\to\inp{n_{i}}$ followed by an active
map $\inp{n_{i}}\to\inp 1$. To describe the action of $F^{\Tw}$
on morphisms, suppose we are given a morphism $\pr{u,v}:\pr{f:\inp n\to\inp m}\to\pr{g:\inp k\to\inp l}$
in $\Tw\pr{\Fin_{\ast}^{\act}}$. The map $F^{\Tw}\pr{u,v}\from\prod_{i=1}^{m}F\inp{n_{i}}\to\prod_{j=1}^{l}F\inp{k_{j}}$
is induced by the composites
\[
\inp{n_{v\pr j}}\rightsquigarrow\inp n\xrightarrow{u}\inp k\mono\inp{k_{j}}.
\]
Here the first map is induced by the inclusion $f^{-1}\pr{v(j)}\hookrightarrow\{1,\dots,n\}$
and the identification of $\underline{n_{v\pr j}}$ with $f^{-1}\pr{v\pr j}$
by an order-preserving map. (In the monoidal case, we use the
category $\nabla$ introduced in Remark \ref{rem:nabla} instead
of $\Fin_{\ast}$ and replace strongly inert maps by inert maps.)

\end{construction}

The category $\Fin_{\ast}^{\act}$ has a symmetric monoidal structure
whose monoidal product is given by the coproduct $\vee$ of based
sets. (Likewise, the category $\nabla$ has a monoidal structure
whose monoidal product ``deletes the maximum of the left factor
and the minimum of the right factor.'') The functor $F^{\Tw}$
of Construction \ref{const:Tw} is symmetric monoidal for this
symmetric monoidal structure and the cartesian monoidal structure
on $\mathsf{RelCat}$, so its Grothendieck construction inherits
a symmetric monoidal structure \cite[$\S$A.1]{moeller_vasilakopoulou_mongroth}.
This gives rise to the following construction:

\begin{construction}
\label{const:Perm}We define a functor 
\[
\opn{Perm}:\Fun\pr{\Fin_{\ast},\sf{RelCat}}\to\sf{PermRelCat}
\]
as follows: Given a functor $F:\Fin_{\ast}\to\sf{RelCat}$, the
underlying relative category $\opn{Perm}\pr F$ is the relative
Grothendieck construction (Example \ref{exa:rel_gr}) of the
functor $F^{\Tw}:\pr{\Tw\pr{\Fin_{\ast}^{\act}}^{\op}}^{\op}\to\sf{RelCat}$.
So a typical object of $\opn{Perm}\pr F$ has the form $\pr{u:\inp n\epi\inp m,X_{1},\dots,X_{m}}$,
where $u$ is an active map of $\Fin_{\ast}$ and each $X_{i}$
is an object of $F\inp{n_{i}}$. The symmetric monoidal structure
is provided by the tensor product
\begin{align*}
 & \pr{u:\inp n\epi\inp m,X_{1},\dots,X_{m}}\otimes\pr{v:\inp k\epi\inp l,Y_{1},\dots,Y_{l}}\\
= & \pr{w:\inp{n+k}\to\inp{m+l},X_{1},\dots,X_{m},Y_{1},\dots,Y_{l}},
\end{align*}
where $w\pr i=u\pr i$ for $1\le i\leq n$ and $w\pr{n+i}=m+v\pr i$
for $1\leq i\leq k$. The unit object is $\pr{\id_{\inp 0},\ast}$,
and the braidings are given by the bijection $\inp n\vee\inp m\cong\inp m\vee\inp n$
that interchanges the summand, together with the identity morphisms
in $F\inp m$ and $F\inp n$. 
\end{construction}

The following proposition identifies the localization of $\Perm\pr F$:
\begin{prop}
\label{prop:Perm_und}For every functor $F:\Fin_{\ast}\to\sf{RelCat}$
satisfying the Segal condition, the inclusion $\{\inp 1\epi\inp 1\}\hookrightarrow\Tw\pr{\Fin_{\ast}}$
induces a DK-equivalence
\[
F\inp 1\xrightarrow{\simeq}\Perm\pr F.
\]
\end{prop}

The proof of Proposition \ref{prop:Perm_und} relies on a lemma:
\begin{lem}
\label{lem:Tw}For any category $\cal C$, the forgetful functor
\[
\Tw\pr{\cal C}\to\cal C
\]
is (homotopy) initial.
\end{lem}

\begin{proof}
We must show that, for each $C\in\cal C$, the category $\Tw\pr{\cal C}\times_{\cal C}\cal C_{/C}$
has contractible classifying space. Unwinding the definitions,
we can identify this category with the Grothendieck construction
of the composite
\[
\pr{\cal C_{/C}}^{\op}\to\cal C^{\op}\xrightarrow{\pr{\cal C_{\bullet/}}^{\op}}\sf{Cat}.
\]
Thus, by Thomason's homotopy colimit theorem \cite[Theorem 1.2]{Tho79},
there is a homotopy equivalence
\[
B\pr{\Tw\pr{\cal C}\times_{\cal C}\cal C_{/C}}\simeq\hocolim_{X\in\pr{\cal C_{/C}}^{\op}}B\pr{\cal C_{X/}}^{\op}.
\]
The right-hand side is contractible, since the classifying spaces
of the slice categories of $\cal C$ are contractible. The proof
is now complete.
\end{proof}
\begin{rem}
In the situation of Lemma \ref{lem:Tw}, Corollary \ref{cor:rel_cart}
shows more strongly that the projection $\Tw\pr{\cal C}\to\cal C$
is a localization.
\end{rem}

\begin{proof}
[Proof of Proposition \ref{prop:Perm_und}]Let $U:\Tw\pr{\Fin_{\ast}^{\act}}\to\Fin_{\ast}^{\act}$
denote the forgetful functor. There is a natural transformation
$F\circ U\To F^{\Tw}$, which is a natural DK-equivalence because
$F$ satisfies the Segal condition. Therefore, it suffices to
show that the relative functor
\[
F\inp 1\to\int_{\Tw\pr{\Fin_{\ast}^{\act}}^{\op}}F\circ U
\]
is a DK-equivalence. This is clear, because the functors $\int_{\Tw\pr{\Fin_{\ast}^{\act}}^{\op}}F\circ U\to\int_{\pr{\Fin_{\ast}^{\act}}^{\op}}F$
and $F\inp 1\to\int_{\pr{\Fin_{\ast}^{\act}}^{\op}}F$ are DK-equivalences
by Lemma \ref{lem:Tw} and Corollary \ref{cor:rel_cart_L}. The
proof is now complete.
\end{proof}

\subsection{\label{subsec:proof_Man10}Proof of Theorem \ref{thm:Man10}}

So far, we have constructed a pair of functors 
\[
\Fact:\sf{PermRelCat}\rl\Fun^{\mathrm{Seg}}\pr{\Fin_{\ast},\sf{RelCat}}:\Perm,
\]
which are relative functors by Propositions \ref{prop:Seg} and
Proposition \ref{prop:Perm_und}. Theorem \ref{thm:Man10} says
that the composites $\Fact\circ\Perm$ and $\Perm\circ\Fact$
are connected by zig-zags of natural weak equivalences, and this
is what we are going to prove in the remainder of this section
(Proposition \ref{prop:PermFact} and \ref{prop:FactPerm}).

\subsubsection{\label{subsubsec:permfact}Comparing $\protect\Perm\circ\protect\Fact\protect\pr{\protect\cal C}$
with $\protect\cal C$}
\begin{construction}
\label{const:Perm_Gamma}We construct a natural transformation
$\Perm\circ\Fact\To\id_{\sf{PermRelCat}}$ of endofunctors on
$\sf{PermRelCat}$. Let $\cal C$ be a permutative category.
We wish to define a functor
\[
\Perm\circ\Fact\pr{\cal C}\to\cal C.
\]
Using the universal property of the Grothendieck construction
as a lax colimit \cite[Theorem 10.2.3]{JY21}, it will suffice
to construct an oplax natural transformation from the functor
$\Fact\pr{\cal C}^{\Tw}:\Tw\pr{\Fin_{\ast}^{\act}}\to\sf{Cat}$
to the constant functor at $\cal C$, and this is what we will
do.

For each object $\inp n\epi\inp m\in\Tw\pr{\Fin_{\ast}^{\act}}$,
we define a functor
\[
\Fact^{\Tw}\pr{\inp n\epi\inp m}=\prod_{1\leq i\leq m}\Fact_{n_{i}}\pr{\cal C}\to\cal C
\]
by $\pr{A^{i}}_{1\leq i\leq m}\mapsto\bigotimes_{i=1}^{m}A^{i}\pr{\u{n_{i}}}$.
Next, for each morphism $\pr{u,v}:\pr{\inp n\epi\inp m}\to\pr{\inp k\epi\inp l}$
in $\Tw\pr{\Fin_{\ast}^{\act}}$, there is a natural transformation
\[\begin{tikzcd}
	{\prod_{1\leq i\leq m }\operatorname{Fact}_{n_i}(\mathcal{C})} \\
	& {\mathcal{C}} \\
	{\prod_{1\leq j\leq l }\operatorname{Fact}_{k_j}(\mathcal{C})}
	\arrow[""{name=0, anchor=center, inner sep=0}, from=1-1, to=2-2]
	\arrow[from=1-1, to=3-1]
	\arrow[""{name=1, anchor=center, inner sep=0}, from=3-1, to=2-2]
	\arrow[shorten <=4pt, shorten >=4pt, Rightarrow, from=1, to=0]
\end{tikzcd}\]whose component at $\pr{A^{i}}_{1\leq i\leq m}\in\prod_{1\leq i\leq m}\Fact_{n_{i}}\pr{\cal C}$
is given by 
\[
\bigotimes_{j=1}^{l}A^{v\pr j}\pr{u_{j}^{-1}\pr{\underline{k_{j}}}}\cong\bigotimes_{i=1}^{m}\bigotimes_{j\in v^{-1}\pr i}A^{i}\pr{u_{j}^{-1}\pr{\underline{k_{j}}}}\to\bigotimes_{i=1}^{m}A^{i}\pr{\u{n_{i}}},
\]
where $u_{j}:\inp{n_{v\pr i}}\to\inp{k_{j}}$ is defined as in
Construction \ref{const:Tw}. These functors and natural transformations
determine an oplax natural transformation $\Fact\pr{\cal C}^{\Tw}\To\cal C$
and hence a functor $\Perm\circ\Fact\pr{\cal C}\to\cal C$. By
construction, this functor is natural in $\cal C\in\sf{PermRelCat}$
and is strictly symmetric monoidal.\footnote{We can define the functor $\Perm\circ\Fact\pr{\cal C}\to\cal C$
even when $\cal C$ is merely a symmetric monoidal category,
but this functor will not be natural on the nose for symmetric
monoidal functors.}
\end{construction}

\begin{prop}
\label{prop:PermFact}For every permutative category $\cal C$,
the functor
\[
\Perm\circ\Fact\pr{\cal C}\to\cal C
\]
of Construction \ref{const:Perm_Gamma} is a DK-equivalence.
\end{prop}

\begin{proof}
Consider the relative functors
\[
\Fact_{1}\pr{\cal C}\xrightarrow{\phi}\Perm\circ\Fact\pr{\cal C}\xrightarrow{\psi}\cal C,
\]
where $\phi$ is induced by the inclusion $\{\inp 1\epi\inp 1\}\hookrightarrow\Tw\pr{\Fin_{\ast}^{\act}}$
and $\psi$ is the functor in question. By Proposition \ref{prop:Seg},
the composite $\psi\circ\phi$ is a DK-equivalence. Also, the
map $\phi$ is a DK-equivalence by Proposition \ref{prop:Perm_und}.
Hence $\psi$ is a DK-equivalence, too.
\end{proof}

\subsubsection{Comparing $\protect\Fact\circ\protect\Perm\protect\pr F$ with
$F$}

We now compare the functor $\Fact\circ\Perm:\opn{Fun}\pr{\Fin_{\ast},\sf{RelCat}}\to\opn{Fun}\pr{\Fin_{\ast},\sf{RelCat}}$
with the identity functor. There does not seem to be a well-behaved
natural transformation between these on the nose, but as we will
see, there is a canonical comparison map between the two if we
post-compose them with the inclusion
\[
\iota:\opn{Fun}\pr{\Fin_{\ast},\sf{RelCat}}\hookrightarrow\opn{OpLaxFun}\pr{\Fin_{\ast},\sf{RelCat}},
\]
where the right-hand side denotes the category of oplax functors
and oplax natural tarnsformations. There is a general technique
to turn an oplax natural transformation into a zig-zag of ordinary
natural transformations. We use this technique to obtain a comparison
between $\Fact\circ\Perm$ and the identity functor of $\opn{Fun}\pr{\Fin_{\ast},\sf{RelCat}}$.
\begin{construction}
\label{const:FactPerm}We define a natural transformation
\[
\eta:\iota\To\iota\circ\Fact\circ\Perm
\]
of functors $\opn{Fun}\pr{\Fin_{\ast},\sf{RelCat}}\to\opn{OpLaxFun}\pr{\Fin_{\ast},\sf{RelCat}}$
as follows. Let $F:\sf{Fin}_{\ast}\to\sf{RelCat}$ be a functor.
For each $\inp n\in\Fin_{\ast}$, there is a functor
\[
\eta_{F,\inp n}:F\inp n\to\Fact_{n}\pr{\Perm\pr F}
\]
carrying each object $X\in F\inp n$ to the $\Fact_{n}$-algebra
defined by
\[
S\mapsto\pr{\inp{\abs S}\epi\inp 1,F\rho^{S}\pr X}.
\]
If $S_{1},\dots,S_{m}\subset\u n$ are pairwise disjoint subsets
with union $S$, then the induced map 
\[
\pr{\inp{\abs S}\epi\inp m,F\rho^{S_{1}}\pr X,\dots,F\rho^{S_{m}}\pr X}\to\pr{\inp{\abs S}\epi\inp 1,F\rho^{S}\pr X}
\]
corresponds to the commutative diagram
\[\begin{tikzcd}
	{\langle |S|\rangle} & {\langle|S| \rangle } \\
	{\langle 1\rangle} & {\langle m\rangle}
	\arrow[equal, from=1-1, to=1-2]
	\arrow[squiggly, from=1-1, to=2-1]
	\arrow[squiggly, from=1-2, to=2-2]
	\arrow[squiggly, from=2-2, to=2-1]
\end{tikzcd}\]and the identity morphisms of the objects $\{F\rho^{S_{i}}\pr X\}_{1\leq i\leq m}$. 

Next, for each morphism $u:\inp n\to\inp m$ in $\Fin_{\ast}$,
the active maps $\{\inp{\abs{u^{-1}\pr S}}\epi\inp{\abs S}\}_{S\subset\u m}$
and the identity morphisms of $\{F\rho^{u^{-1}\pr S}\pr X\}_{S\subset\u m}$
determine a natural transformation depicted as 
\[\begin{tikzcd}
	{F\langle n\rangle} & {\operatorname{Fact}_n (\operatorname{Perm}(F))} \\
	{F\langle m\rangle} & {\operatorname{Fact}_m (\operatorname{Perm}(F)).}
	\arrow[""{name=0, anchor=center, inner sep=0}, "{\eta_{F,\langle n\rangle}}", from=1-1, to=1-2]
	\arrow["Fu"', from=1-1, to=2-1]
	\arrow["{(\operatorname{Fact}\circ \operatorname{Perm}(F))u}", from=1-2, to=2-2]
	\arrow[""{name=1, anchor=center, inner sep=0}, "{\eta_{F,\langle m\rangle}}"', from=2-1, to=2-2]
	\arrow[shorten <=4pt, shorten >=4pt, Rightarrow, from=1, to=0]
\end{tikzcd}\]These natural transformations determine an oplax natural transformation
$\eta_{F}:F\To\Fact\circ\Perm\pr F$, natural in $F$.
\end{construction}

To turn the oplax natural transformations $F\To\Fact\circ\Perm\pr F$
into a zig-zag of ordinary natural transformations, we use the
following path construction:
\begin{construction}
\label{const:path}Let $f:\cal X\to\cal Y$ be a relative functor
of relative categories. We let $\Path\pr f$ denote the fiber
product
\[
\cal X\times_{\Fun\pr{\{0\},\cal Y}}\Fun^{{\rm weq}}\pr{[1],\cal Y},
\]
where $\Fun^{{\rm weq}}\pr{[1],\cal Y}$ denotes the full subcategory
of $\Fun\pr{[1],\cal Y}$ spanned by the weak equivalences. We
regard $\Path\pr f$ as a relative category whose weak equivalences
are those whose images in $\cal X$ and $\Fun\pr{\{1\},\cal Y}$
are weak equivalences. 
\end{construction}

\begin{rem}
\label{rem:path}In the situation of Construction \ref{const:path},
the projection $\Path\pr f\to\cal X$ is a DK-equivalence, since
it admits a left adjoint $X\mapsto\pr{X,\id_{f\pr X}}$ whose
unit and counit are natural weak equivalences.
\end{rem}

\begin{construction}
\label{const:lax->ord}Let $\cal C$ be an ordinary category,
let $F,G:\cal C\to\sf{RelCat}$ be functors, and let $\alpha:F\To G$
be an oplax natural transformation $\alpha:F\To G$. We define
a functor $\Path\pr{\alpha}:\cal C\to\sf{RelCat}$ as follows:
\begin{itemize}
\item On objects, we have $\opn{Path}\pr{\alpha}\pr C=\Path\pr{\alpha_{C}}$.
\item If $f:C\to D$ is a morphism in $\cal C$, then the functor $\opn{Path}\pr{\alpha_{C}}\to\Path\pr{\alpha_{D}}$
carries an object $\pr{X,u:\alpha_{C}\pr X\to Y}$ to 
\[
\pr{Ff\pr X,\alpha_{D}\circ Ff\pr X\to Gf\circ\alpha_{C}\pr X\xrightarrow{Gf\pr u}Gf\pr Y},
\]
where the map $\alpha_{D}\circ Ff\pr X\to Gf\circ\alpha_{C}\pr X$
is the structure map of the oplax natural transformation $\alpha$.
\item The action of $\Path\pr{\alpha}$ on morphisms is defined so
that the assignment $\pr{X,u:\alpha_{C}\pr X\to Y}\mapsto\pr{X,Y}$
determines a natural transformation $\Path\pr{\alpha}\To F\times G$.
\end{itemize}
Note that the assignment $\alpha\mapsto\pr{\Path\pr{\alpha}\To F\times G}$
determines a functor
\begin{align*}
 & \Fun\pr{[1],\opn{OpLaxFun}\pr{\cal C,\sf{RelCat}}}\times_{\Fun\pr{\{0\}\amalg\{1\},\opn{OpLaxFun}\pr{\cal C,\sf{RelCat}}}}\Fun\pr{\{0\}\amalg\{1\},\Fun\pr{\cal C,\sf{RelCat}}}\\
\to & \Fun\pr{[1],\Fun\pr{\cal C,\sf{RelCat}}}.
\end{align*}
\end{construction}

\begin{prop}
\label{prop:FactPerm}There are natural transformations
\[
\id_{\Fun\pr{\Fin_{\ast},\sf{RelCat}}}\stackrel[\simeq]{\alpha}{\Leftarrow}\Path\pr{\eta_{\bullet}}\stackrel{\beta}{\Rightarrow}\Fact\circ\Perm
\]
of endofunctors $\Fun\pr{\Fin_{\ast},\sf{RelCat}}$, with the
following properties:
\begin{enumerate}
\item For every functor $F:\Fin_{\ast}\to\sf{RelCat}$, the natural
transformation $\alpha_{F}:\Path\pr{\eta_{F}}\To F$ is a natural
DK-equivalence.
\item For every functor $F:\Fin_{\ast}\to\sf{RelCat}$ satisfying the
Segal condition, the natural transformation $\beta_{F}:\Path\pr{\eta_{F}}\To\Gamma\circ\Perm\pr F$
is a natural DK-equivalence.
\end{enumerate}
\end{prop}

\begin{proof}
The natural transformations $\alpha$ and $\beta$ are obtained
by applying Construction \ref{const:lax->ord} to the components
of the natural transformation $\eta$ of Construction \ref{const:FactPerm}.
Property (1) follows from Remark \ref{rem:path}. We complete
the proof by proving (2).

Suppose we are given a functor $F$ as in (2). We must show that
$\beta_{F}:\Path\pr{\eta_{F}}\To\Fact\circ\Perm\pr F$ is a natural
DK-equivalence. By part (1), the functor $\Path\pr{\eta_{F}}$
satisfies the Segal condition. By Proposition \ref{prop:Seg},
the functor $\Fact\circ\Perm\pr F$ also satisfies the Segal
condition. Therefore, it suffices to show that the functor
\[
\beta_{F,\inp 1}:\Path\pr{\eta_{F}}\inp 1\to\pr{\Fact\circ\Perm\pr F}\inp 1
\]
is a DK-equivalence. For this, we observe that the functor $\alpha_{F,\inp 1}:\Path\pr{\eta_{F}}\inp 1\to F\inp 1$
has a section $\sigma_{F,\inp 1}:F\inp 1\to\Path\pr{\eta_{F}}\inp 1$
satisfying $\beta_{F,\inp 1}\circ\sigma_{F,\inp 1}=\eta_{F,\inp 1}$.
It will therefore suffice to show that the map
\[
\eta_{F,\inp 1}:F\inp 1\to\pr{\Fact\circ\Perm\pr F}\inp 1
\]
is a DK-equivalence. According to Proposition \ref{prop:Seg},
there is a DK-equivalence $\pr{\Fact\circ\Perm\pr F}\inp 1\to\Perm\pr F$,
so it suffices to show that the composite
\[
F\inp 1\xrightarrow{\eta_{F,\inp 1}}\pr{\Fact\circ\Perm\pr F}\inp 1\to\Perm\pr F=\int_{\Tw\pr{\Fin_{\ast}^{\act}}^{\op}}F^{\Tw}
\]
is a DK-equivalence. But this is the content of Proposition \ref{prop:Perm_und},
and we are done.
\end{proof}

\section{\label{sec:main}Main Result}

We now state and prove the main theorem of this paper (Theorem
\ref{thm:main}). 
\begin{notation}
\label{nota:loc}We let $L:\sf{SMRelCat}\to\cal{SM}\Cat_{\infty}$
denote the functor which is characterized by the following universal
property: There is a natural (in $\pr{\cal C,\cal W}\in\sf{SMRelCat}$
and $\cal D^{\t}\in\cal{SM}\Cat_{\infty}$) equivalence
\[
\Fun^{\t}\pr{L\pr{\cal C,\cal W},\cal D}\simeq\Fun^{\t,\cal W}\pr{\cal C,\cal D},
\]
where $\Fun^{\t,\cal W}\pr{\cal C,\cal D}\subset\Fun^{\t}\pr{\cal C,\cal D}$
denotes the full subcategory spanned by the symmetric monoidal
functors $\cal C^{\t}\to\cal D^{\t}$ carrying each morphism
in $\cal W$ to an equivalence in $\cal D$.
\end{notation}

\begin{thm}
\label{thm:main}The functor $L:\sf{SMRelCat}\to\cal{SM}\Cat_{\infty}$
induces a categorical equivalence
\[
\sf{SMRelCat}[{\rm DK}^{-1}]\xrightarrow{\simeq}\cal{SM}\Cat_{\infty},
\]
where the left-hand side denotes the localization at the symmetric
monoidal relative functors whose underlying functors are DK-equivalences.
\end{thm}

\begin{proof}
Since symmetric monoidal categories are functorially equivalent
to permutative categories \cite[Proposition 4.2]{May78}, the
inclusion $\sf{PermRelCat}\hookrightarrow\sf{SMRelCat}$ is a
homotopy equivalence of relative categories. It will therefore
suffice to show that the functor $\sf{PermRelCat}[{\rm DK}^{-1}]\to\cal{SM}\Cat_{\infty}$
is a categorical equivalence.

Define a (ordinary) category $\sf{RelCocart}^{{\rm SM}}\pr{\Fin_{\ast}}$
as follows: Its objects are relative cocartesian fibrations $\cal E\to\sf{Fin}_{\ast}$
whose induced cocartesian fibration $\cal E[{\rm weqfib}^{-1}]\to\Fin_{\ast}$
(Proposition \ref{prop:fiberwise_loc}) is a symmetric monoidal
$\infty$-category. A morphism $\pr{p:\cal E\to\Fin_{\ast}}\to\pr{q:\cal E'\to\sf{Fin}_{\ast}}$
is a relative functor $f:\cal E\to\cal E'$ satisfying $qf=p$
and which carries $p$-cocartesian morphisms to $q$-cocartesian
morphisms. Using Proposition \ref{prop:fiberwise_loc}, we can
extend $L$ to a functor $L^{{\rm fib}}:\sf{RelCocart}^{\mathrm{SM}}\pr{\Fin_{\ast}}\to\cal{SM}\Cat_{\infty}$
which is characterized by the natural equivalence
\[
\Fun^{\t}\pr{L^{{\rm fib}}\pr{\cal E},\cal D^{\t}}\simeq\Fun_{\Fin_{\ast}}^{{\rm cocart},{\rm fib}}\pr{\cal E,\cal D^{\t}}.
\]
(The right-hand side is defined as in Notation \ref{nota:Funcart}.)
Proposition \ref{prop:Seg} gives us a natural transformation
\[
\int_{\Fin_{\ast}}\Fact\pr -\to\int_{\Fin_{\ast}}\pr -^{\bullet}=\pr -^{\t}
\]
of functors $\sf{PermRelCat}\to\sf{RelCocart}^{{\rm SM}}\pr{\Fin_{\ast}}$.
By Corollary \ref{cor:rel_cart}, this natural transformation
becomes a natural equivalence when composed with the functor
$L^{{\rm fib}}$. Therefore, it suffices to show that the composite
\[
\sf{PermRelCat}\xrightarrow{\Fact}\Fun^{{\rm Seg}}\pr{\Fin_{\ast},\sf{RelCat}}\xrightarrow{\int_{\Fin_{\ast}}}\sf{RelCocart}^{{\rm SM}}\pr{\Fin_{\ast}}\xrightarrow{L^{{\rm fib}}}\cal{SM}\Cat_{\infty}
\]
is a localization at DK-equivalences. Thus, in light of Theorem
\ref{thm:Man10}, we are reduced to showing that the composite
$L^{{\rm fib}}\circ\int_{\Fin_{\ast}}$ is a localization at
DK-equivalences. Using Proposition \ref{prop:fiberwise_loc},
we deduce that the diagram 
\[\begin{tikzcd}
	{\operatorname{Fun}^{\mathrm{Seg}}(\mathsf{Fin}_\ast,\mathsf{RelCat})} & {\mathsf{RelCocart}^{\mathrm{SM}}({\mathsf{Fin}_\ast})} \\
	{\operatorname{Fun}^{\mathrm{Seg}}(\mathsf{Fin}_\ast,\mathcal{C}\mathsf{at}_\infty)} & {\mathcal{SMC}\mathsf{at}_\infty}
	\arrow["{\int_{\mathsf{Fin}_\ast}}", from=1-1, to=1-2]
	\arrow["{\operatorname{Fun}(\mathsf{Fin}_\ast,L)}"', from=1-1, to=2-1]
	\arrow["{L^{\mathrm{fib}}}", from=1-2, to=2-2]
	\arrow["\simeq"', from=2-1, to=2-2]
\end{tikzcd}\]commutes up to natural equivalence, where the bottom horizontal
arrow is the unstraightening equivalence and $L:\sf{RelCat}\to\Cat_{\infty}$
is defined in Definition \ref{def:loc}. Therefore, it suffices
to show that the left vertical arrow is a localization at DK-equivalences.
This is the content of Corollary \ref{cor:FunSeg}, and we are
done.
\end{proof}

\section{\label{sec:appl}Application}

As an application of Theorem \ref{thm:main}, we give a short
proof of the following theorem of Nikolaus and Sagave \cite{NS17}.
Note that the same argument works in the monoidal case, too.
\begin{thm}
\label{thm:NS17}For every presentably symmetric monoidal $\infty$-category
$\cal C^{\t}$, there is a left proper, combinatorial, simplicial
symmetric monoidal model category (with cofibrant unit) whose
underlying symmetric monoidal $\infty$-category is equivalent
to $\cal C$.
\end{thm}

\begin{proof}
As explained in \cite[Proposition 2.4]{NS17}, it is sufficient
to consider the case where $\cal C^{\t}=\cal P\pr{\cal A}^{\t}$
for some small symmetric monoidal $\infty$-category $\cal{A^{\t}}$,
where $\cal P\pr{\cal A}^{\t}$ denotes the symmetric monoidal
$\infty$-category of space-valued presheaves on $\cal A$, with
symmetric monoidal structure given by the Day convolution. Using
Theorem \ref{thm:main}, we can find a symmetric monoidal category
$\cal A_{0}$ and a symmetric monoidal localization functor $L:\cal A_{0}^{\t}\to\cal A^{\t}$.
The induced functor $\cal P\pr{\cal A_{0}}\to\cal P\pr{\cal A}$
is an accessible localization, so as explained in \cite[Proposition 2.3]{NS17},
it suffices to show that $\cal P\pr{\cal A_{0}}^{\t}$ is presented
by a left proper, combinatorial, simplicial symmetric monoidal
model category. This is clear: For instance, consider the model
category $\Fun\pr{\cal A_{0},\SS}$ of simplicial presheaves,
equipped with the projective model structure. This model category
is left proper, combinatorial, and simplicial. It is also symmetric
monoidal with respect to the Day convolution, as can be checked
on the level of generating cofibrations and generating acyclic
cofibrations \cite[Theorem 11.6.1]{Hirschhorn}. It is clear
from the universal property of the Day convolution monoidal structure
\cite[Corollary 4.8.1.12]{HA} that this symmetric monoidal model
category presents $\cal P\pr{\cal A_{0}}$, and we are done.
\end{proof}
\appendix

\section{\label{sec:relcat}Relative Categories}

In this section, we recall and establish a few key results on
the homotopy theory of relative categories. 
\begin{defn}
\label{def:loc}We define a functor $L:\sf{RelCat}\to\Cat_{\infty}$
so that there is a natural equivalence
\[
\Fun\pr{L\pr{\cal C,\cal W},\cal D}\simeq\Fun^{\cal W}\pr{\cal C,\cal D},
\]
where $\Fun^{\cal W}\pr{\cal C,\cal D}$ denotes the full subcategory
spanned by the functors $\cal C\to\cal D$ that carry morphisms
in $\cal W$ to equivalences in $\cal D$.
\end{defn}

In \cite{BK12b}, Barwick and Kan essentially showed that the
functor $L:\sf{RelCat}\to\Cat_{\infty}$ is a localization at
the DK-equivalences, using simplicial categories as models of
$\pr{\infty,1}$-categories. For later applications, we provide
a version of this result using the model category $\SS^{+}$
of \textit{marked simplicial sets}, introduced in \cite[Chapter 3]{HTT}.
\begin{thm}
\label{thm:BK12_sset}The inclusion $\iota:\sf{RelCat}\hookrightarrow\SS^{+}$
is a homotopy equivalence of relative categories (Definition
\ref{def:relcat}).
\end{thm}

\begin{proof}
To avoid confusions, we \textit{will }make a distinction between
categories and their nerves for the duration of the proof. Thus
$\iota$ is given by $\iota\pr{\cal C,\cal W}=\pr{N\pr{\cal C},\mor\cal W}$.

Let $\sf{Cat}^{+}$ denote the category of pairs $\pr{\cal C,W}$,
where $\cal C$ is a category and $W$ is a set of morphisms
of $\cal C$ containing all identity morphisms. Morphisms $\pr{\cal C,W}\to\pr{\cal D,W'}$
are functors $\cal C\to\cal D$ that carry $W$ into $W'$. The
nerve functor determines a fully faithful functor $\sf{Cat}^{+}\to\sf{sSet}^{+}$,
and we regard $\sf{Cat}^{+}$ as a relative category whose weak
equivalences are the maps whose images in $\SS^{+}$ are weak
equivalences. The inclusion $\sf{RelCat}\to\sf{Cat}^{+}$ is
a homotopy equivalence, so it suffices to show that the inclusion
$\widetilde{\iota}:\sf{Cat}^{+}\to\sf{sSet}^{+},\,\pr{\cal C,W}\mapsto\pr{N\pr{\cal C},W}$
is a homotopy equivalence.

For each simplicial set $X$, let $\Del_{/X}=\Del\times_{\SS}\SS_{/X}$
denote the category of simplices of $X$. There is a map $\varepsilon_{X}:N\pr{\Del_{/X}}\to X$
of simplicial sets which carries an $m$-simplex $\Delta^{n_{0}}\to\cdots\to\Delta^{n_{m}}\xrightarrow{\sigma}X$
to the $m$-simplex $\Delta^{m}\xrightarrow{f}\Delta^{n_{m}}\xrightarrow{\sigma}X$,
where the map $f$ carries each vertex $i\in\Delta^{m}$ to the
image of $n_{i}\in\Delta^{n_{i}}$. Given a set $S$ of edges
of $X$, we let $M_{\pr{X,S}}$ denote the set of morphisms $\alpha$
of $\Del_{/X}$ satisfying at least one of the following conditions:
\begin{enumerate}
\item The edge $\varepsilon_{X}\pr{\alpha}$ is degenerate.
\item The morphism $\alpha$ has the form $\Delta^{\{0\}}\hookrightarrow\Delta^{1}\xrightarrow{\sigma}X$
for some $\sigma\in S$.
\end{enumerate}
We then define a functor
\[
D:\SS^{+}\to\sf{Cat}^{+}
\]
by $D\pr{X,S}=\pr{N\pr{\Del_{/X}},M_{\pr{X,S}}}$. The map $\varepsilon_{X}$
induces a map $\varepsilon_{\pr{X,S}}:\widetilde{\iota}\circ D\pr{X,S}\to\pr{X,S}$
of marked simplicial sets. To complete the proof, it suffices
to show that $\varepsilon_{\pr{X,S}}$ is a weak equivalence.

By construction, there are pushout diagrams of marked simplicial
sets 
\[\begin{tikzcd}
	{\coprod_S(\Delta^1)^\flat} & {(N(\mathbf{\Delta}_{/X}),M_{X^\flat})} & {X^\flat} \\
	{\coprod_S(\Delta^1)^\sharp} & {(N(\mathbf{\Delta}_{/X}),M_{(X,S)})} & {(X,S).}
	\arrow[from=1-1, to=1-2]
	\arrow[from=1-1, to=2-1]
	\arrow["{\varepsilon_{X^\flat}}", from=1-2, to=1-3]
	\arrow[from=1-2, to=2-2]
	\arrow[from=1-3, to=2-3]
	\arrow[from=2-1, to=2-2]
	\arrow["{\varepsilon_{(X,S)}}", from=2-2, to=2-3]
\end{tikzcd}\]These squares are homotopy pushouts since every marked simplicial
set is cofibrant and the vertical arrows are cofibrations. Therefore,
it suffices to show that $\varepsilon_{X^{\flat}}$ is a weak
equivalence. But this is the content of Joyal's delocalization
theorem \cite[Theorem 1.3]{Ste2017}, and the proof is complete.
\end{proof}
\begin{cor}
\label{cor:FunSeg}The post-composition by the functor $L:\sf{RelCat}\to\Cat_{\infty}$
induces a categorical equivalence
\[
\Fun^{{\rm Seg}}\pr{\Fin_{\ast},\sf{RelCat}}[{\rm DK}^{-1}]\xrightarrow{\simeq}\Fun^{{\rm Seg}}\pr{\Fin_{\ast},\Cat_{\infty}}.
\]
\end{cor}

\begin{proof}
By Theorem \ref{thm:BK12_sset}, the inclusion
\[
\Fun^{{\rm Seg}}\pr{\Fin_{\ast},\sf{RelCat}}\to\Fun^{{\rm Seg}}\pr{\Fin_{\ast},\SS^{+}}
\]
is a homotopy equivalence of relative categories. Here the right
hand denotes the full subcategory of $\Fun\pr{\Fin_{\ast},\SS^{+}}$
spanned by the functors $F:\Fin_{\ast}\to\SS^{+}$ satisfying
the Segal condition, i.e., for each $n\geq0$, the functor $F\inp n\to\prod_{1\leq i\leq n}F\inp 1$
is a weak equivalence of marked simplicial sets. We can extend
$L$ to a functor $\widetilde{L}:\SS^{+}\to\Cat_{\infty}$ characterized
by a universal property as in Definition \ref{def:loc}, and
it suffices to show that $\widetilde{L}$ induces a categorical
equivalence
\[
\Fun^{{\rm Seg}}\pr{\Fin_{\ast},\SS^{+}}[{\rm DK}^{-1}]\xrightarrow{\simeq}\Fun^{{\rm Seg}}\pr{\Fin_{\ast},\Cat_{\infty}}.
\]

The projective model structure on $\Fun\pr{\Fin_{\ast},\SS^{+}}$
admits a Bousfield localization whose fibrant objects are the
projectively fibrant objects satisfying the Segal condition \cite[Proposition A.3.7.3]{HTT}.
It follows from \cite[Theorem 7.5.30]{HCHA} that the functor
$\Fun^{{\rm Seg}}\pr{\Fin_{\ast},\SS^{+}}[{\rm DK}^{-1}]\to\Fun\pr{\Fin_{\ast},\SS^{+}}[{\rm DK}^{-1}]$
is fully faithful. Therefore, it suffices to show that the functor
\[
\Fun\pr{\Fin_{\ast},\SS^{+}}[{\rm DK}^{-1}]\to\Fun\pr{\Fin_{\ast},\Cat_{\infty}}
\]
is an equivalence of $\infty$-categories. This follows from
\cite[Proposition 4.2.4.4]{HTT} (or \cite[Theorem 7.9.8]{HCHA}).
\end{proof}

\section{\label{sec:rel_cart_fib}Relative Cartesian Fibrations}

In this section, we establish a few results on relative versions
of cartesian fibrations.
\begin{defn}
\label{def:rel_cart}A \textbf{relative cartesian fibration}
consists of the following data:
\begin{enumerate}
\item A cartesian fibration $p:\cal E\to\cal C$ of $\infty$-categories.
\item For each object $C\in\cal C$, the structure of a relative $\infty$-category
(Definition \ref{def:relcat}) on the fiber $p^{-1}\pr C=\cal E_{C}$.
\end{enumerate}
These data are required to satisfy the following condition: For
each morphism $C\to D$ in $\cal C$, the induced functor $\cal E_{D}\to\cal E_{C}$
preserves weak equivalences. 

If $p:\cal E\to\cal C$ is a relative cartesian fibration, we
will write ${\rm weqfib}={\rm weqfib}\pr p$ for the set of weak
equivalences of the fibers of $p$. We will regard $\cal E$
as a relative category whose weak equivalences are the morphisms
that can be written as a composite of a weak equivalence in a
fiber of $p$, followed by a $p$-cartesian morphism.
\end{defn}

\begin{rem}
There is an obvious dual notion of \textbf{relative cocartesian
fibrations}, and we freely use this notion in the main body of
the paper.
\end{rem}

\begin{example}
\label{exa:rel_gr}Let $\cal C$ be an ordinary category, and
let $F:\cal C^{\op}\to\sf{RelCat}$ be a pseudofunctor. If $U:\sf{RelCat}\to\sf{Cat}$
denotes the forgetful functor, the Grothendieck construction
$\int\pr{U\circ F}\to\cal C$ has a natural structure of a relative
cartesian fibration. We denote the resulting relative category
by $\int F$ and refer to it as the \textbf{relative Grothendieck
construction} of $F$.
\end{example}

\begin{notation}
\label{nota:Funcart}Let $p:\cal X\to\cal C$ and $q:\cal Y\to\cal C$
be cartesian fibrations of $\infty$-categories. We let $\Fun_{\cal C}^{{\rm cart}}\pr{\cal X,\cal Y}\subset\Fun_{\cal C}\pr{\cal X,\cal Y}$
denote the full subcategory spanned by the functors that preserve
cartesian morphisms over $\cal C$. In the case where $p$ is
equipped with the structure of a relative cartesian fibration,
we let $\Fun_{\cal C}^{{\rm cart},{\rm fib}}\pr{\cal X,\cal Y}\subset\Fun_{\cal C}^{{\rm cart}}\pr{\cal X,\cal Y}$
denote the full subcategory spanned by the functors $f:\cal X\to\cal Y$
such that, for each $C\in\cal C$, the functor $f_{C}:\cal X_{C}\to\cal Y_{C}$
carries weak equivalences to equivalences.
\end{notation}

\begin{prop}
\label{prop:fiberwise_loc}Let 
\[\begin{tikzcd}
	{\mathcal{X}} && {\mathcal{Y}} \\
	& {\mathcal{C}}
	\arrow["f", from=1-1, to=1-3]
	\arrow["p"', from=1-1, to=2-2]
	\arrow["q", from=1-3, to=2-2]
\end{tikzcd}\]be a commutative diagram of $\infty$-categories. Suppose that:
\begin{enumerate}
\item The functor $p$ is a relative cartesian fibration.
\item The functor $q$ is a categorical fibration.
\item The functor $f$ carries weak equivalences in the fibers of $p$
to equivalences in $\cal Y$.
\end{enumerate}
Then the following conditions are equivalent:

\begin{enumerate}[label=(\alph*)]

\item The functor $q$ is a cartesian fibration, the functor
$f$ preserves cartesian morphisms over $\cal C$, and for each
$C\in\cal C$, the functor $\cal X_{C}\to\cal Y_{C}$ is a localization
at the weak equivalences.

\item The functor $f$ is a localization at ${\rm fibweq}\pr p$.

\item The functor $q$ is a cartesian fibration, the functor
$f$ belongs to $\Fun^{{\rm cart},{\rm fib}}\pr{\cal X,\cal Y}$,
and for each cartesian fibration $\cal Z\to\cal C$, the functor
$f$ induces a categorical equivalence
\[
\Fun_{\cal C}^{{\rm cart}}\pr{\cal Y,\cal Z}\xrightarrow{\simeq}\Fun_{\cal C}^{{\rm cart},{\rm fib}}\pr{\cal X,\cal Z}.
\]

\end{enumerate}
\end{prop}

\begin{proof}
We first show that (a)$\iff$(b). The implication (a)$\implies$(b)
is proved in \cite[\href{https://kerodon.net/tag/02LW}{Tag 02LW}]{kerodon}.
For the converse, suppose that condition (b) is satisfied. Use
the straightening--unstraightening equivalence to factor $p$
as a composite $\cal X\xrightarrow{f'}\cal Y'\xrightarrow{q'}\cal C$,
where $\pr{f',q'}$ satisfies condition (a). Using the implication
(a)$\implies$(b), we deduce that $f'$ is a localization at
${\rm fibweq}\pr p$. Therefore, there is a categorical equivalence
$g:\cal Y'\xrightarrow{\simeq}\cal Y$ over $\cal C$, such that
$g'f$ is equivalent to $f$ as a functor over $\cal C$. In
particular, for each $C\in\cal C$, the functor $f_{C}:\cal X_{C}\to\cal Y_{C}$
is equivalent to the composite $\cal X_{C}\xrightarrow{f'_{C}}\cal Y'_{C}\xrightarrow[\simeq]{g_{C}}\cal Y_{C}$.
This implies that $f_{C}$ is a localization at the weak equivalences,
proving that (b)$\implies$(a).

Next, we show that (b)$\implies$(c). Suppose that (b) is satisfied.
Using the implication (b)$\implies$(a), we find that $q$ is
a cartesian fibration and $f$ belongs to $\Fun^{{\rm cart},{\rm fib}}\pr{\cal X,\cal Y}$.
The universal property of $q$ appearing in (c) is immediate
from the universal property of localizations. Hence (b)$\implies$(c).

We complete the proof by showing (c)$\implies$(b). Suppose that
(c) is satisfied. Factor $p$ as $\cal X\xrightarrow{f'}\cal Y'\xrightarrow{q'}\cal C$,
where $f'$ is a localization at ${\rm fibweq}\pr p$ and $q'$
is a categorical fibration. Using the implication (b)$\implies$(c),
we can find a categorical equivalence $g:\cal Y\xrightarrow{\simeq}\cal Y'$
such that $gf$ is naturally equivalent to $f'$. This means
that $f'$ is a localization at ${\rm fibweq}\pr p$, so that
(c)$\implies$(b) as claimed.
\end{proof}
Proposition \ref{prop:fiberwise_loc} has the following corollaries:
\begin{cor}
\label{cor:rel_cart}Consider a commutative diagram 
\[\begin{tikzcd}
	{\mathcal{E}} && {\mathcal{E}'} \\
	& {\mathcal{C}}
	\arrow["f", from=1-1, to=1-3]
	\arrow["p"', from=1-1, to=2-2]
	\arrow["{p'}", from=1-3, to=2-2]
\end{tikzcd}\]of $\infty$-categories satisfying the following conditions:
\begin{enumerate}
\item The functors $p$ and $p'$ are relative cartesian fibrations.
\item The functor $f$ is relative (but may not preserve cartesian
morphisms).
\item For each $C\in\cal C$, the functor $L\pr{\cal E_{C}}\to L\pr{\cal E'_{C}}$
is a categorical equivalence.
\end{enumerate}
Then the functor $\cal E[{\rm weqfib}\pr p^{-1}]\to\cal E'[{\rm weqfib}\pr{p'}^{-1}]$
is a categorical equivalence.
\end{cor}

\begin{proof}
Using hypotheses (1) and (2) and the implication (b)$\implies$(c)
of Proposition \ref{prop:fiberwise_loc}, we obtain a commutative
diagram 
\[\begin{tikzcd}
	{\mathcal{E}[\mathrm{weqfib}(p)^{-1}]} && {\mathcal{E}'[\mathrm{weqfib}(p')^{-1}]} \\
	& {\mathcal{C},}
	\arrow["{\overline{f}}", from=1-1, to=1-3]
	\arrow["{\overline{p}}"', from=1-1, to=2-2]
	\arrow["{\overline{p'}}", from=1-3, to=2-2]
\end{tikzcd}\]where $\overline{f},\overline{p},\overline{p'}$ are functors
induced by $f,p,p'$, and $\overline{p}$ and $\overline{p'}$
are cartesian fibrations. Condition (2) implies that $\overline{f}$
preserves cartesian morphisms. Condition (3) and Proposition
\ref{prop:fiberwise_loc} imply that $\overline{f}$ induces
a categorical equivalence between the fibers of $\overline{p}$
and $\overline{p'}$. Hence $\overline{f}$ is a categorical
equivalence, as claimed.
\end{proof}
\begin{cor}
\label{cor:rel_cart_L}Let $p:\cal E\to\cal C$ be a relative
cartesian fibration, and let $f:\cal D\to\cal C$ be an initial
functor of $\infty$-categories. The functor
\[
L\pr{\cal E\times_{\cal C}\cal D^{{\rm }}}\to L\pr{\cal E}
\]
is a categorical equivalence.
\end{cor}

\begin{proof}
Set $\cal E'=\cal E\times_{\cal C}\cal D$, and let $p':\cal E'\to\cal D$
denote the projection. We also factor $p$ as $\cal E\xrightarrow{i}\cal E[{\rm weqfib}\pr p^{-1}]\xrightarrow{q}\cal C$,
where $i$ is a localization at ${\rm fibweq}\pr p$ and $q$
is a categorical fibration. We will write $q':\cal E[{\rm weqfib}\pr p^{-1}]\times_{\cal C}\cal D\to\cal D$
for the projection.

By the implication (a)$\implies$(b) of Proposition \ref{prop:fiberwise_loc},
the functor
\[
\cal E'[{\rm weqfib}\pr{p'}^{-1}]\to\cal E[{\rm weqfib}\pr p^{-1}]\times_{\cal D}\cal C
\]
is a categorical equivalence. Therefore, it suffices to show
that the functor
\[
\pr{\cal E[{\rm weqfib}\pr p^{-1}]\times_{\cal D}\cal C}[{\rm cart}\pr{q'}^{-1}]\to\cal E[{\rm weqfib}\pr p^{-1}][{\rm cart}\pr q^{-1}]
\]
is a categorical equivalence. Since $f$ is initial, this follows
from \cite[Lemma 3.3.4.1]{HTT}.
\end{proof}

\subsection*{Acknowledgment}

The author acknowledges the work of Michael Mandell \cite{Man10},
which this paper builds on extensively. He thanks the anonymous
referees for valuable comments, and is grateful to Maxime Ramzi
and Tobias Lenz for helpful discussions. He also appreciates
Daisuke Kishimoto and Mitsunobu Tsutaya for constant support
and encouragement. This work was supported by JSPS KAKENHI Grant
Number 24KJ1443.

\providecommand{\bysame}{\leavevmode\hbox to3em{\hrulefill}\thinspace}
\providecommand{\MR}{\relax\ifhmode\unskip\space\fi MR }
\providecommand{\MRhref}[2]{%
  \href{http://www.ams.org/mathscinet-getitem?mr=#1}{#2}
}
\providecommand{\href}[2]{#2}

\end{document}